\documentclass{amsart}
\usepackage[utf8]{inputenc}

\usepackage{multirow, longtable, amsmath, amssymb, amsthm}
\usepackage[mathscr]{euscript}
\usepackage[margin = 1 in]{geometry}
\usepackage{cite, fancyref}
\usepackage{graphicx}
\usepackage{caption}
\usepackage{float}
\usepackage{listings}
\usepackage{url}
\usepackage{bbm}

\newtheorem{thm}{Theorem}
\newtheorem{lem}[thm]{Lemma}
\newtheorem{prop}[thm]{Proposition}
\newtheorem{defn}{Definition}
\newtheorem{coro}{Corollary}

\newcommand{\N}{{\mathbb{N}}}
\newcommand{\Z}{\mathbb{Z}}

\newcommand{\A}{{\mathcal{A}}}

\title{On the Convergence of the Density of Terras' Set}
\author{I. Assani, E. Ebbighausen}
\address{Idris Assani, University of North Carolina at Chapel Hill}
\email{assani@email.unc.edu}
\address{Ethan Ebbighausen, University of North Carolina at Chapel Hill}
\email{ejebbigh@email.unc.edu}
 \thanks{The second author was supported by the Honors Carolina Excellence Fund}
\begin{document}

\begin{abstract}
The Collatz Conjecture's connection to dynamical systems opens it to a variety of techniques aimed at recurrence and density results. First, we turn to density results and strengthen the result of Terras through finding a strict rate of convergence. This rate gives a preliminary result on the Triangle Conjecture, which describes a set nodes that would dominate $L^{C} = \{y \in \N \, | \, T^{k}(y) > y , \,\forall k \in \N  \}$. Second, we extend prior arguments to show that the construction of several classes of measures imply the bounded trajectories piece of the Collatz Conjecture. 
\end{abstract}

\maketitle

\section{Introduction}

The Collatz Map $T: \N \to \N$ given by 
\begin{equation}
    T(x) = 
    \begin{cases} 
    \frac{x}{2} & \text{$x$ even}   \\
    \frac{3x-1}{2} & \text{$x$ odd}
    \end{cases}
\end{equation} 
poses difficult questions on recurrence due to its complex behavior. The famous Collatz conjecture (or Syracuse conjecture, Kakutani Conjecture, Ulam's Conjecture) states that for any $n \in \N$, $T^{k}(n) = 1$ for some $k \in \N$, or that the Collatz map returns to the cycle $\{1,2\}$ for all values. This may be broken into two pieces: first that every value returns to a cycle, or the bounded trajectories conjecture, and second that the only cycle is $\{1,2\}$. 

\bigskip

It was shown in \cite{Assanipre} that the natural numbers may be broken into 3 components, $\N = C \cup D_1 \cup D_2$, where $C$ is the set of all elements of cycles of the Collatz map, $D_1$ is the set of all values returning to a cycle, and $D_2$ is the set of all nodes which do not return to a cycle. The bounded trajectories conjecture then says $D_2 = \emptyset$ and the unique cycle conjecture states $C = \{1,2\}$. The same paper gives a critereon for the former conjecture, that $D_2$ is empty if and only if there exists a finite measure $\mu$ defined on every value in $\N$ which is everywhere nonzero and power bounded with respect to the map $T$. This result was extended in \cite{SyracuseNonsing} to a more general class of maps.

\bigskip

Furthermore, \cite{SyracuseNonsing} investigated another property of the inverse Collatz map. Recall that for any point $a \in \N$, we may consider all preimages of $a$ given by $\cup_{i=0}^{\infty} T^{-i}(a)$, which may be referred to as the inverse tree generated by $a$ due to the organization of these values discussed in \cite{SyracuseNonsing}. A further extension of this set may be considered.

\begin{defn}
    Let $a \in \N$. The Chain-Tree generated by $a$ is $\cup_{i=0}^{\infty} \cup_{j=0}^{\infty} T^{-j}(T^{i}(a))$. A chain in this set is a collection of elements indexed by the integers, $\{a_{z}\}_{z \in \Z}$ where $T(a_z) = a_{z+1}$. 
\end{defn}

\begin{figure}[ht]
\centering
\includegraphics[scale=0.5]{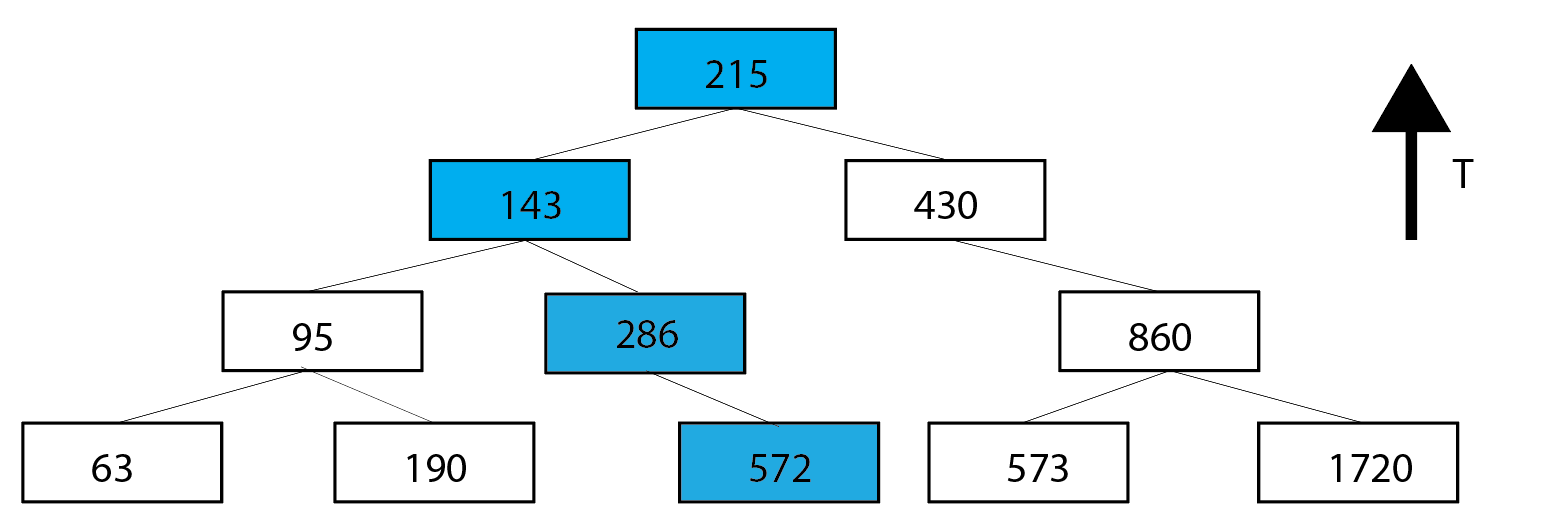}
\caption{A Snippet of a Chain-Tree with the Chain section highlighted in Blue}
\label{fig:one}
\end{figure}
\bigskip

The chain tree is stable under both the forward and inverse Collatz map, and any two nodes within the same chain tree will generate the same chain tree under this definition. The chain is an arbitrary choice indexing the ``levels" of the chain-tree (see Figure \ref{fig:one}), and certain chains, such as the ``leftmost" chain discussed in \cite{SyracuseNonsing} have structural properties. However, this labeling is enough to weaken the requirements on a measure to show that $D_2$ is nonempty. See section \ref{sec:measequiv}.

\bigskip

The same paper left an open question related to the structure of the inverse ``triangle" first posed by I. Assani. Consider the partitions of $\N$ by mod 3 remainders $\N_0, \N_1, \N_2$. For $a_0 \in \N_0 \cup \N_1$, $T^{-1}(a_0) = \{ 2a_0 \}$, and for $a_2 \in \N_2$, $T^{-1}(a_2) = \{ 2a_2, \frac{2a_2-1}{3} \}$. This offers an immediate arrangement of the inverse images. Further, each $\N_2$ node $a_2$ is of the form $3^{k}h-1$ where $h$ is not a multiple of $3$, $k \geq 1$. Using the above formula, $T^{-1}(3^{k}h-1) = \{2(3^{k}h-1), 3^{k-1}(2h)-1\}$, so that there exists a sequence of $k$ preimages $3^{k-1}(2h)-1, 3^{k-2}(2^{2}h)-1, ... ,2^{k}h-1$ which is strictly decreasing. The open conjecture states that this is the only such sequence of preimages in $\bigcup_{i=0}^{k} T^{-i}(3^{k}h-1)$. We call this the Triangle Conjecture. Precisely, fix a $\N_2$ node $3^{k}h-1$, where $k,h \geq 1$. Then, the triangle conjecture says that there exists no $a \in T^{-k}(3^{k}h-1)$ such that $a \neq 2^{k}h-1$ and $a \leq T^{l}(a)$ for all $1 \leq a \leq k$. 

Recall the set $L = \{ y \in \N \, | \, \exists k \in \N \,\, \text{such that} \,\, T^{k}(y) < y \}$ of Terras \cite{Terras1976}. The Triangle Conjecture creates a set of nodes which dominates $L^{c}$, locating possible nodes with no lesser image. The two concepts are closely intertwined, and the topics in \ref{sec:introdens} strengthen the result of Terras by finding a rate of convergence of the density of this set to $1$. This rate of convergence then provides an upper bound to the number of possible nodes in $a\in T^{-k}(3^{k}h-1)$ which have no lesser image in $\{a,T(a),...,T^{k}(a)\}$, providing a partial result on the triangle conjecture.

\begin{figure}[ht]
    \centering
    \includegraphics[scale=0.5]{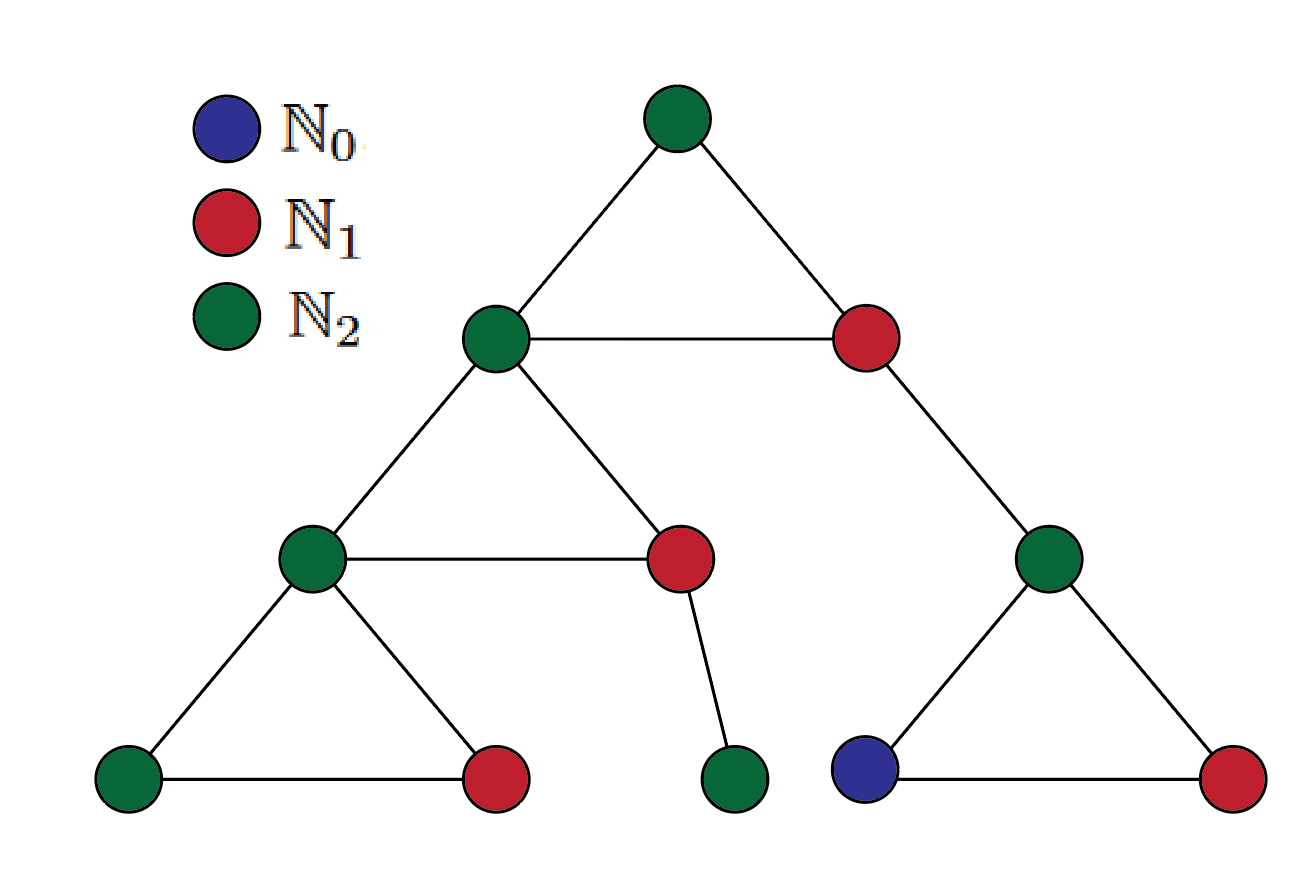}
    \caption{The Triangle for $k=3$}
    \label{fig:two}
\end{figure}
\section{An Introduction to Density Results}
\label{sec:introdens}
In lieu of results on strict cases, it is worthwhile to at least obtain some information on sets with full density in the natural numbers. An important foundational case for the Collatz map is the set $L = \{ y \in \N \, | \, \exists k \in \N \,\, \text{such that} \,\, T^{k}(y) < y \}$ introduced by R.Terras \cite{Terras1976}, which Terras proved to have density $1$. This result was extended to show the set that $M_{c} = \{ y \in \N \, | \, \exists k \in \N \,\, \text{such that} \,\, T^{k}(y) < y^{c} \}$ has density $1$ for $c \geq \log_{3}(2)$ by I.Korec \cite{Korec1994}. Recently, T.Tao has also shown that the set $D_2$ has logarithmic density $1$ \cite{TTao}. In this section, we shall focus on introducing an alternate proof of Terras' original result which gives more precisely the rate of convergence of the density of the set $L$ to $1$. First, we recall some structures original to Terras.

\medskip
\begin{defn}
    For $k,y \in \N$, define $E_{k}(y)$ to be the vector of length $k$ whose $i^{th}$ component is $1$ if $T^{i-1}(y)$ is odd, and $0$ if it is even. Define $S_{k}(y)$ to be the sum of the elements of $E_{k}(y)$. 
\end{defn}

Also, we use from \cite{Terras1976} the following proposition.

\begin{prop}
\label{prop:period}
    $E_{k}(y) = E_{k}(x)$ if and only if $x \equiv y \mod 2^{k}$. 
\end{prop}

This proposition is often called periodicity for the Collatz map, as it allows consideration of a finite number of numbers which may be extrapolated to a density. With these, we may introduce a refinement of Terras' result. 

\begin{thm}
    For fixed $k \in \N$, let $L_{k} = \{ y \in \N \, | \, \exists \, m, \,\, 1 \leq m \leq k , \,\, \text{such that} \,\, T^{m}(y) \leq y \}$. The density of $L_{k}^{C}$ is at most \begin{equation}\frac{2^{m}}{2^{k}}\prod_{n=0}^{m} \frac{2n+1}{n+1}\end{equation} where $m =  \lfloor \frac{k}{2}\rfloor $.
\end{thm}
\begin{proof}
To prove this result, we first construct a restriction on the $E_{k}$ vectors which any such $a$ must follow, then to count the number of these we must consider a form of a ``delayed phase Pascal's triangle" and compute the row sums.

\medskip

Consider $y \in L_{k}^{C}$. Then,  we have that $\frac{T^{n}(y)}{y} > 1$ for all $1 \leq n \leq k$. Set $l = S_{n}(y)$ and and in particular that for all such $n$ 
\begin{equation}
    (\frac{3y+1}{2y})^{l}(\frac{y}{2y})^{n-l} \geq\frac{T^{n}(y)}{T^{n-1}(y)} \frac{T^{n-1}(y)}{T^{k-2}(y)} ... \frac{T(y)}{y} = \frac{T^{n}(y)}{y} \geq 1
\end{equation} 
By taking logs, we note that 
\begin{equation}
    S_{n}(y) \geq \frac{ n \ln(2)}{\ln(3+\frac{1}{y})} 
\end{equation}
so that $S_{n}(y) \geq \frac{n}{2}$ in general since $y \geq 1$. By the periodicity of Proposition \ref{prop:period}, $y \in L_{k}^{C}$ implies there exists a unique value at most $2^{k}$ with the same $E_{k}$ vector which satisfies this inequality. The number of vectors which satisfy this inequality, divided by $2^{k}$, then bounds the density of $L_{k}^{C}$ from above.

\medskip
We now construct the modified version of Pascal's Triangle, which is inspired by counting the $E_{n}$ vectors satisfying the above inequality. 
    
First, take a map $\tau: \mathbb{N} \rightarrow \mathbb{N}$. We define a sequence of sequences, $\{ \{x^{n}_{i}\}_{i \geq 0} \}_{n \geq 0}$ where $x_{0}^{n} = 1$ for all $n$ and for $n > 1$, 
\begin{equation}
    x_{k}^{n} = 
    \begin{cases} \displaystyle{\sum_{m = \tau(k) - 1}^{n-1}} x_{k-1}^{m}  \hspace{2mm} & n-1 \geq \tau(k)-1 \\ 
    & \\
    0 \hspace{2mm} & \text{else} \\
    \end{cases}
\end{equation}
Consider the $n^{th}$ sequence to correspond to the $n^{th}$ row of the constructed triangle. For example, if we take $\tau(k) = k$, then this defines the standard Pascal Triangle. The function $\tau$ restricts when the rows may expand to have more nonzero values in the sequence. Consider, for example, the first 11 rows of the triangle constructed by $\tau(k) = 2k$ (starting at the $0^{th}$ row). 

\begin{align*}
     i=&0 \,\,\,\,\,  1 \,\,\,\,\, 2 \,\,\,\,\,\, 3 \,\,\,\,\,\,\, 4 \,\,\,\,\,\,\, 5 ... \\
    n=0 \hspace{8mm} &1 \,\,\,\, 0 \,\,\,\,\, 0 \,\,\,\,\, 0 ... \\
    n=1 \hspace{8mm}&1 \,\,\,\, 0 \,\,\,\,\, 0\,\,\,\,\, 0 ... \\
    n=2 \hspace{8mm}&1 \,\,\,\, 1 \,\,\,\,\,0 \,\,\,\,\, 0 ...\\ 
    n=3 \hspace{8mm}&1 \,\,\,\, 2 \,\,\,\,\, 0 \,\,\,\,\, 0 ...\\ 
    n=4 \hspace{8mm}&1 \,\,\,\, 3 \,\,\,\, 2 \,\,\,\,\,0 \,\,\,\,\, 0 ...\\
    n=5 \hspace{8mm}&1 \,\,\,\, 4 \,\,\,\, 5\,\,\,\,\, 0 \,\,\,\,\, 0 ...\\ 
    n=6 \hspace{8mm}&1 \,\,\,\, 5 \,\,\,\, 9 \,\,\,\,\,\, 5 \,\,\,\,\,0 \,\,\,\,\, 0 ...\\
    n=7 \hspace{8mm}&1 \,\,\,\, 6 \,\,\,\, 14 \,\,\,\, 14\,\,\,\,\, 0 \,\,\,\,\, 0 ...\\
    n=8 \hspace{8mm}&1 \,\,\,\, 7 \,\,\,\, 20 \,\,\,\, 28 \,\,\,\, 14 \,\,\,\,\,0 \,\,\,\,\, 0 ...\\
    n=9 \hspace{8mm}&1 \,\,\,\, 8 \,\,\,\, 27 \,\,\,\, 48 \,\,\,\, 42 \,\,\,\,\,0 \,\,\,\,\, 0 ...\\ 
    n=10 \hspace{8mm}&1 \,\,\,\, 9 \,\,\,\, 35 \,\,\,\, 75 \,\,\,\, 90 \,\,\,\, 42 \,\,\,\,\,0 \,\,\,\,\, 0 ...\\  
    &\vdots
\end{align*}
This generated triangle may also be considered row wise, where $x^{n}_{i} = x^{n-1}_{i} + x^{n-1}_{i-1}$ for $ i \leq \tau(n)$,  $x^{n}_{0}$ is still fixed as $1$, and all other elements are $0$. 

\bigskip
We now connect this construction to the problem. Consider the number of possible $E_{n}(y)$ vectors which satisfy $S_{n}(y) \geq \frac{n}{2}$. We may construct a recurrence relation by counting the number of $0$s possible in the vector. There is always only a single vector with no $0$s, the vector of all $1$s. If the number of $0$s, $l$, is less than $\frac{n}{2}$, then when considering a $E_{n+1}$ vector, we may take a vector with $l$ zeroes and attach another $0$ or add a $1$ at the end while still satisfying this inequality. Therefore, the number of $E_{n+1}$ vectors with $l$ zeroes, $1 \leq l \leq \frac{n}{2}$, is the number of $E_{n}$ vectors with $l$ zeroes plus the number with $l-1$ zeroes. Now, consider that if $l > \frac{n}{2}$ and if $l \leq \frac{n+1}{2}$, the number of $E_{n+1}$ vectors with $l$ zeroes is precisely the number of $E_{n}$ vectors with $l-1$. For all other values, there are $0$ vectors satisfying the relations. Tracing these recurrence relations shows that the number of $E_{n}$ vectors with $i$ zeroes is precisely $x^{n}_{i}$ for $\tau(k) = 2k$. Therefore, counting these vectors reduces to taking the row sum of the triangle constructed by this $\tau$. Fix the $\{\{x^{n}_{i}\}_{i \geq 0} \}_{n \geq 0}$ as those generated this way. 
    
\bigskip

Notice that that row sums are strictly increasing, so it suffices to consider only the odd-number rows to generate an upper bound. From the triangle above, this amounts to considering the rows

\begin{align*}
    n=1 \,\,\,\, &1 \,\,\,\, \\ 
    n=3 \,\,\,\,&1 \,\,\,\, 2 \\ 
    n = 5 \,\,\,\,&1 \,\,\,\, 4 \,\,\,\, 5 \\ 
    n=7 \,\,\,\,&1 \,\,\,\, 6 \,\,\,\, 14 \,\,\,\, 14 \\
    n=9 \,\,\,\,&1 \,\,\,\, 8 \,\,\,\, 27 \,\,\,\, 48 \,\,\,\, 42 \\ 
&\vdots
\end{align*}
Consider now this triangle within its own right. Define, now, a sequence of sequences for this triangle $\{\{y^{n}_{i}\}_{i \geq 0} \}_{n \geq 0}$ where $y^{n}_{i} = x^{2n+1}_{i}$. 
Considering the recurrence relation row-wise,  $y^{n}_{k} = y^{n-2}_{k-2} + 2y^{n-2}_{k-1} + y^{n-2}_{k}$. It is shown in \cite{SHAPIRO197683} that for the largest index $k$ so that $y^{n}_{k} \neq 0$, $y^{n}_{k}$ is precisely the $n^{th}$ Catalan number, $\frac{1}{n+1} \binom{2n}{n}$.

\bigskip
Then, induction may be applied to the triangle on $y^{n}_{i}$ to see that the row sum of the triangle on $x^{n}_{i}$ at level $2n+1$ is 
\begin{equation}
\label{eq:num}
    \sum_{i} x^{2n+1}_{i} = \sum_{i} y^{n}_{i} = 2^{n} \prod_{k=0}^{n} \frac{2k+1}{k+1}
\end{equation}
This is easy to see in the case $n=0$ or $n=1$. Let it be shown for values up to $n-1$ and consider row $n$. Then, also note that 
\begin{equation}
    \sum_{i} y^{n}_{i} = \sum_{i} y^{n-1}_{i-2} + 2y^{n-1}_{i-1} + y^{n-1}_{i} = 4 (\sum_{i}y^{n-1}_{i}) - y^{n-1}_{\frac{n-1}{2}} = 4 (\sum_{i}y^{n-1}_{i}) - y^{n-1}_{k}
\end{equation}
where $\frac{n-1}{2} = k$ is the largest index $i$ so $y^{n-1}_{i} = x^{2n-1}_{i}$ is nonzero. Now, we may apply the inductive assumption and the result from \cite{SHAPIRO197683} to note that this sum is 
\begin{equation}
    4(2^{n-1} \prod_{k=0}^{n-1} \frac{2k+1}{k+1}) - \frac{1}{n+1} \binom{2n}{n}
\end{equation}
With some algebraic manipulation, 
\begin{align}
    \frac{1}{n+1} \binom{2n}{n}  = \frac{1}{n+1} \left( \frac{(2n)!}{n!n!} \right) =& \frac{1}{n+1} \left( (\frac{2}{1} \times \frac{4}{2} \times ... \times \frac{2n}{n}) ( \frac{1}{1} \times \frac{3}{2} \times ... \times \frac{2n-1}{n}) \right)  \\
    =& \frac{2}{n+1}\left(2^{n-1} \prod_{k=0}^{n-1} \frac{2k+1}{k+1}\right)
\end{align} 
Therefore, the $n^{th}$ row sum is $$(2 - \frac{1}{n+1})(2^{n} \prod_{k=0}^{n-1} \frac{2k+1}{k+1})  = 2^{n} \prod_{k=0}^{n} \frac{2k+1}{k+1}$$ which corresponds to the $2n+1^{st}$ row sum  $\sum_{i} x^{2n+1}_{i}$ as desired.

\bigskip
Now, consider an arbitrary $k$ row sum of the $x^{n}_{i}$ to bound above. If $k$ is odd, then we have calculated the row sum for $k$ by considering $m = \frac{k-1}{2} = \lfloor \frac{k}{2} \rfloor$ and applying it to the above formula to find the sum of the $y^{m}_{i}$. If $k$ is even, then $k+1$ is odd and $x^{k+1}_{i} \geq x^{k}_{i}$ for all $i$, so considering $m = \frac{(k+1)-1}{2} = \frac{k}{2}$ and using the above formula then bounds the row sum above. 

\bigskip
Finally, applying periodicity, the density of $L_{k}^{C}$ is then at-most \begin{equation}\frac{2^{m}}{2^{k}}\prod_{n=0}^{m} \frac{2n+1}{n+1}\end{equation} where $m =  \lfloor \frac{k}{2}\rfloor $.
    
\end{proof}

\begin{coro}
    The density of $L$ is $1$.
\end{coro}
\begin{proof}
    Since $L^{C} \subset L_{k}^{C}$ for all $k$, we have that the density of $L^{C}$ is at most 
    \begin{equation}
    \displaystyle{\lim_{k \to \infty}} \frac{2^{m}}{2^{k}}\prod_{n=0}^{m} \frac{2n+1}{n+1} = \displaystyle{\lim_{k \to \infty}} \frac{2^{m}}{2^{k-m-1}}\prod_{n=0}^{m} \frac{2n+1}{2n+2} \leq \displaystyle{\lim_{m \to \infty}} 2(\prod_{n=0}^{m} \frac{2n+1}{2n+2})
    \end{equation}
    for $m$ a function of $k$ as above. Note then that the rightmost limit is then 
    \begin{equation}
        \exp(\displaystyle{\lim_{m \to \infty}} \ln(2) + \sum_{n=0}^{m} \ln(1-\frac{1}{2n+2})
    \end{equation}
    for $x < 1$, we have that $ln(1-x) \leq -x$ since $f(x) = \ln(1-x) + x$ has $f(0) = 0$ and $f'(x) = \frac{-x}{1-x} \leq 0$. Thus, since the exponential is an increasing function, this limit is at most 
    \begin{equation}
        \exp(\displaystyle{\lim_{m \to \infty}} \ln(2) + \sum_{n=0}^{m} -\frac{1}{2n+2})
    \end{equation}
    The series is harmonic and thus diverges to negative infinity, so that the total limit is $0$.
\end{proof}

\subsection{Bounding the Number of Failures within the Triangle}

Recall from \cite{SyracuseNonsing} the conjecture on the inverse image triangle of a $\N_2$ node. Precisely, it is conjectured that given any $k,n \in \N$, $a \in T^{-k}(3^{k}h-1)$, $a \neq 2^{k}h-1$, then there exists some $m$ so $1 \leq m \leq k$ such that $T^{k}(a) < a$. Since we may consider $3^{k-1}(3h)-1$ and further reductions, this is to say that this holds for all nodes in the preimages $\bigcup_{i=1}^{k}T^{-i}(3^{k}h-1)$ not of the form $2^{a}3^{k-a}h-1$, i.e. the leftmost branch as distinguished in that paper as well. 

\medskip

This result would locate the values not in $L$ as defined in section $3$, by generating a set in which, if values in $L^{c}$ exist, they must be located. The main theorem of the prior section introduces an upper bound for the number of such exception cases.

\begin{coro}
For a fixed $k,h \in \N$, the number of $a \in T^{-k}(3^{k}h-1)$ such that for all $m$ so $1 \leq m \leq k$ has $T^{m}(a) > a$ is at most 
\begin{equation}
2^{m}\prod_{n=0}^{m} \frac{2n+1}{n+1}
\end{equation} where $m =  \lfloor \frac{k}{2}\rfloor $.
\end{coro}
\begin{proof}
    The number of nodes $a$ in the triangle generated by $3^{k}h-1$ such that $T^{m}(a) > a$ for $1 \leq m \leq k$ is at most the number of $E_{k}(y)$ vectors corresponding to $1 \leq y \leq 2^{k}$ so $T^{m}(y) > y$ for $1 \leq m \leq k$. Thus, this is precisely equation \ref{eq:num}.
\end{proof}

We may also simplify the number of cases to calculate drastically by examining the structure of the triangle and how it varies across different values of $h$.

\begin{prop}
The structure of the triangle generated by $3^{k}h-1$ is invariant with respect to $h$. That is to say, for each $a \in T^{-l}(3^{k}h-1)$, $1 \leq l \leq k$, and for any $h_1 \in \N$, there exists $a_1 \in T^{-l}(3^{k}h_1 - 1)$ such that $E_{l}(a) = E_{l}(a_1)$. 
\end{prop}

\begin{proof}
Fix $k,h \in \N$ and consider $a \in T^{-l}(3^{k}h-1)$ for $1 \leq l \leq k$. We wish to show that there exist $\alpha,\beta$ not dependent on $h$ so that $a = 2^{l}3^{k-l}h\alpha + \beta$. This decomposes $a$ into the part maintaining the initial power of $3$ ($2^{l}3^{k-l}\alpha$) and the ``remainder" part that helps locate it on the branch $\beta$.  Indeed, consider that the preimage of $3^{k}h-1$ is $\{2(3^{k-1})h-1, 2(3^{k}h) - 2\}$. Working inductively, if we assume there are $\alpha_0$ and $\beta_0$ not dependent on $h$ so $T(a) = 2^{l-1}3^{k-l+1}h\alpha_0 + \beta_0$, then $a \in \{2^{l}3^{k-l}h\alpha_0 + \frac{2\beta_0-1}{3}, 2^{l}3^{k-l}h(3\alpha_0) + 2\beta_0\}$ where, when the first option is possible, $\frac{2\beta_0-1}{3}$ is an integer. Therefore, we may express $a$ in the desired form as well. 

\bigskip

Now, consider $E_{l}(a)$. Denote, for arbitrary $h_1 \in \N$, $a_1 = 2^{l}3^{k-l}h_1\alpha + \beta$. By Terras' periodicity result, $E_{l}(a_1) = E_{l}(a)$, and further this implies by repeated applications of the Collatz map that $T^{l}(a_1) = 3^{k}h_1 -1$. 
    
\end{proof}

\begin{prop}
    If there exists an $a \in T^{-k}(3^{k}h-1)$ for which all $1 \leq m \leq k$ have $T^{m}(a) > a$, then the corresponding $a_1 \in T^{-k}(3^{k}-1)$ has the same property. 
\end{prop}
\begin{proof}
Since $E_{k}(a) = E_{k}(a_1)$, we have that the way $T$ acts on both values is the same within the triangle. Then, we need only compare the ratio $\frac{T^{m}(a_1)}{a_1}$. Note that this is precisely 
\begin{equation} 
\frac{T^{m}(a_1)}{T^{m-1}(a_1)} \frac{T^{m-1}(a_1)}{T^{m-2}(a_1)}  \,\,...\,\, \frac{T(a_1)}{a_1}
\end{equation}

Now, we need only show $\frac{T^{l}(a_1)}{T^{l-1}(a_1)} \geq \frac{T^{l}(a)}{T^{l-1}(a)}$ for each $1 \leq l \leq m$. Without loss of generality, we examine $l=1$. Let $a = 2^{k}h\alpha + \beta$ so $a_1 = 2^{k}\alpha + \beta$. We then have that $T(a_1)$ is either $\frac{2^{k}\alpha + \beta}{2}$ or $\frac{3(2^{k}\alpha + \beta)-1}{2}$. For the first case, \begin{equation}
    \frac{T(a)}{a} = \frac{2^{k}h\alpha + \beta}{2(2^{k}h\alpha + \beta)} = \frac{2^{k}\alpha + \beta}{2(2^{k}\alpha + \beta)} = \frac{T(a_1)}{a_1}
\end{equation}
For the second 
\begin{equation}
    \frac{T(a_1)}{a_1} = \frac{3(2^{k}\alpha + \beta)-1}{2(2^{k}\alpha + \beta)} \geq \frac{3(2^{k}h\alpha + \beta)-1}{2(2^{k}h\alpha + \beta)} \geq \frac{3}{2} - \frac{1}{4} >1
\end{equation} 
since $h \geq 1$.
    
\end{proof}

Proposition 11 reduces solving the triangle conjecture to proving that it holds in all cases for $h=1$, and so the structure to check is more strictly defined. Further, the invariance with respect to $h$ occurs beyond $h \in \N$, and considering an extension of the Collatz map to $T: \Z \to \Z$ allows consideration of $h=0$ or to looking at the triangle of $k$ levels generated by $-1$, which has the same structure as for $h=1$ well.It is not yet clear how this extension connects to the triangle conjecture, but is an interesting point nonetheless. 

\section{Measure Equivalences}
\label{sec:measequiv}

In \cite{Assanipre}, it was shown that bounding the Collatz Trajectories is equivalent to constructing a measure which is finite, power-bounded with respect to the Collatz map, and everywhere nonzero. The construction was extended further to general maps including at least a single cycle in \cite{SyracuseNonsing}. Within the case of the Collatz map, there arise difficulties in the construction of such a measure due to family-chain connections also discussed in \cite{SyracuseNonsing} which allow values of preimages of a given point to vary widely in a hard-to-predict way. In this section, we develop further the theory surrounding such applications of measures to bounding the trajectories of the Collatz Map by extending to the cases of weaker measures, and further use a tool from dynamical systems to also gain information on the length of cycles possible depending on current knowledge of the minimum bound of cycle elements for cycles other than $\{1,2\}$. 

\bigskip

\subsection{Limiting Measure Cases}

The principle desire in extending the measures is  to weaken the requirement on bounding the measure of given sets, thus extending the measures to take advantage of properties noticed in \cite{Assanipre}. The following two results show that we may trade these requirements for behaviors over time of the measure, perhaps allowing for leveraging long-term decreases or increases of the values of preimages. For this section, assume each measure is non-negative and defined on the $\sigma$-algebra $P(\N)$. The following establishes the idea behind this extension.

\bigskip

\begin{prop}
Let there exist a finite measure $\mu$ on $\N$ such that $\displaystyle{\lim_{n \to \infty}} \mu(T^{-n}(A))$ exists for all $A \subset \N$. Then, $D_2$ must have measure 0. 
\end{prop}
\begin{proof}

By contradiction, let $a \in D_2$ have $\mu(a) > 0$. Let $E = \bigcup_{i=0}^{\infty} \bigcup_{j=0}^{\infty} T^{-j}(T^{i}(a))$ be a chain-tree in $D_2$. We may then pick a chain $H = \{a_{z}\}_{z \in \Z}$ in $E$ to be a set such that $T(a_{z}) = a_{z+1}$ and $a_0 = a$. We focus on the measure on $H$. First, we construct a set $\A_n$. Define $A_n = \{a_{z}\}_{z \in n\Z}$ to be every n-th node in the chain. We pick  $\A_n = \bigcup_{i=1}^{\infty} T^{-in}(A_n)$.  For each $\A_n$, the $n$ different sets $\A_n, T^{-1}(\A_n), T^{-2}(\A_n), ... T^{-(n-1)}(\A_n)$ are disjoint, and $T^{-n}(\A_n) = \A_n$. These sets repeat as we take preimages. Therefore, if any two of these sets have different measures under $\mu$, say $\A_n$ and $T^{-m}(\A_n)$, then the sub-sequences of $T^{-k}(\A_n)$ corresponding to these generated by the $n^{th}$ and $n+m^{th}$ indices converge to different values and the proof is complete.

\bigskip

Next, assume that $\mu(\A_n) = \mu( T^{-1}(\A_n)) =  ... =\mu(T^{-(n-1)}(\A_n))$ for all $n$. Let $\mu(E) = M$. Since $\A_n \cup T^{-1}(\A_n) \cup ... \cup T^{-(n-1)}(\A_n) = E$, we then have that $\mu(\A_{n}) = \frac{M}{n}$. Consider the set $B = \{a_z|  \,\, |z| = 2^{n} \, \text{for} \, n\in \mathbb{N}\}$. Then, there exists a subsequence of $\{T^{-i}(B)\}_{i \in \N}$ given by $\{T^{-i_{k}}(B)\}$ such that $T^{-i_k}(B)$ contains $a_0$ for each $i_{k}$. This shows that there exists a subsequence of the $\{T^{-i}(B)\}_{i \in \N}$ where the limit of the measures of the subsequence is positive.

\bigskip

However, consider that for any $n$,  $B \subset \{a_{z} \, | \, z \in  2^{n}\mathbb{Z}\} \cup \{a_{-2^{n-1}}, a_{-2^{n-2}}, ... ,a_{2^{n-1}}\}$. Note that since $\mu$ is a finite measure and $E$ is an infinite subset of $D_2$, for any finite $S \subset E$, the preimages of $S$ are disjoint and $\lim_{m \to \infty} \mu(T^{-m}(S)) = 0$. Take $ S=\{a_{-2^{n-1}}, a_{-2^{n-2}}, ... ,a_{2^{n-1}}\}$.  Further, it is assumed that $\mu(\A_{2^{n}}) = \mu(T^{-1}(\A_{2^{n}})) = ... = \mu(T^{-(2^{n}-1)}(\A_{2^{n}}) = 2^{-n}M$, and so because $T^{-k}( \{a_{z} \, | \, z \in  2^{n}\mathbb{Z}\})$ is a subset of one of these $2^{n}$ sets, 
\begin{equation}
    \displaystyle{\lim_{k \to \infty}} \mu(T^{-k}(B)) \leq 2^{-n}M
\end{equation}
Since we picked $n$ arbitrarily, the limit then must be $0$. The two subsequences converge to different values, giving a contradiction.
\end{proof}

\bigskip

The central part of the above argument is collapsing the structure of the chain-tree in such a way that its structure mirrors the integers. We may instead expand this to collapse a single level of the chain-tree to a set $B_{z}$ corresponding to the element of the chain $a_z$. This more directly reduces considerations of the chain-tree to considerations of the chain (or the integers), where the Collatz map acts as a right shift. This argument allows us to extend the result further to the more-desirable Cesaro limits as opposed to standard limits, which presents them in a form more common to Ergodic Theory.

\bigskip

\begin{defn}
    Let $(\Omega, \mathcal{F}, \nu)$ be a probability measure space and $V: \Omega \to \Omega$ a $\mathcal{F}$-measurable map. Then, we say $\nu$ is asymptotically mean stationary with respect to $V$ if $\displaystyle{\lim_{N \to \infty}} \frac{1}{N} \sum_{n=1}^{N} \nu(V^{-n}(A))$ exists for all $A \in \mathcal{F}$. 
\end{defn}

\begin{prop}

Let there exists a finite measure $\mu$ on $\N$ such that $\displaystyle{\lim_{k \to \infty}} \,\, \frac{1}{k} \sum_{i=1}^{k} \mu(T^{-i}(A))$ exists for all $A \subset \N$, that is to say that $(\N, P(\N), \mu, T)$ is asymptotically mean stationary. Then, $D_2$ must have measure 0. 

\end{prop}
\textbf{Remark}: Note that $\frac{1}{k}\sum_{i=1}^{k}\mu(T^{-i}(\A_n))$ converges to $\frac{1}{n}$ as $k \to \infty$. In other words, this is a much weaker requirement on the measure than the previous case. The limit does act similarly on finite sets. For any set $A$ such that $\sum_{i=1}^{\infty} \mu(T^{-i}(A)) = l < \infty$, $\frac{1}{k} \sum_{i=1}^{k} \mu(T^{-i}(A)) \leq \frac{l}{k}$ shows that the limit of these means is $0$. Since the above property holds for singletons, it holds for finite sets as well. The effort is then extending to the infinite case in the same way.

\begin{proof}
By contradiction, let $\mu$ be such a measure and $a \in D_2$ be a point so $\mu(a) > 0$. Let $H = \{a_{z} \}_{z \in Z}$ and $E$ be as in the proof of proposition 1, and assume $\mu(E) = 1$ by renormalization. 

\medskip

We begin by redefining a set similar to the $\A_{n}$ in concept. Let $B_{z} = \bigcup_{i=0}^{\infty} T^{-i}(T^{i}(a_{z}))$, so that the $B_z$ correspond to a ``level" of the chain-tree $E$ as demarcated by the $a_z$.

\medskip

First, we construct a set $B$. To begin, select $N$ such that $\sum_{i=N+1}^{\infty} \mu(B_i) + \mu(B_{-i}) < \frac{1}{20}$, so that also $\sum_{-N}^{N} \mu(B_i) \geq \frac{19}{20}$ (these values are mostly arbitrary choices, though the first must be sufficiently small for the following argument). Let $K = 2N+1$. Let $B = [B_{N+1} \cup B_{N+2}\cup ... \cup B_{3N+1}] \cup [B_{7N+4}\cup B_{7N+5}\cup ... \cup B_{11N+5}] \cup ... \cup [B_{17N+9}\cup B_{17N+10}\cup ... \cup B_{23N+11} ]  \cup ...$ To construct $B$ in words move forward to $B_{N}$, take the first $K$ levels $B_{N+1} ... B_{3N+1}$, leave the following $2K$ levels, take the next $2K$ levels $B_{7N+4}, ... , B_{11N+5}$, leave the following $3K$ levels, take the next $3K$ levels, leave the next $6k$ levels ad infinitum. This corresponds to each step following a given pattern. Considering levels as starting from level $B_{-N-1}$, we skip $K$. Then, at each step, we take enough nodes so that the total number of levels included divided by the number of levels since $-N-1$ is $\frac{1}{2}$, then exclude enough that this drops to $\frac{1}{3}$. This allows us to construct two sequences based on these choices which converge in different ways. Consider first the value $\frac{1}{2K}\sum_{i=1}^{2K} \mu(T^{-i}(B))$. Since $T^{-1}(B_i) = B_{i-1}$, and since we move $2K$, we have that the $K$ levels in $B$ given by $B_{N+1},...,B_{3N+1}$ each take the values of $B_{-N},...,B_{N}$ precisely once in the sum $\sum_{i=1}^{2K} \mu(T^{-i}(B))$. Therefore, 
\begin{equation}
    \sum_{i=1}^{2K} \mu(T^{-i}(B)) \geq \left( \frac{19}{20}\right) \cdot K
\end{equation}

Using similar equations, we may track preimages as they pass over the central mass of $\frac{19}{20}$ or stay within the tail mass of $\frac{1}{20}$ to construct two sequences with masses bounded as 

\begin{align}
    &\frac{1}{2K}\sum_{i=1}^{2K} \mu(T^{-i}(B)) \geq \frac{19}{20}\left(\frac{K}{2K}\right) = \frac{19}{40}  
    &\frac{1}{3K}\sum_{i=1}^{3K} \mu(T^{-i}(B)) \leq \frac{1}{20} + \frac{19}{20}\left(\frac{K}{3K}\right)  \\ 
    &\frac{1}{6K}\sum_{i=1}^{5K} \mu(T^{-i}(B)) \geq \frac{19}{20}\left(\frac{3K}{5K}\right) \geq \frac{19}{40}  
    & \frac{1}{9K}\sum_{i=1}^{9K} \mu(T^{-i}(B)) \leq \frac{1}{20} + \frac{19}{20}\left(\frac{3}{9}\right) \\ 
    &\frac{1}{12K}\sum_{i=1}^{12K} \mu(T^{-i}(B)) \geq \frac{19}{20}\left(\frac{6K}{12K}\right)=  \frac{19}{40}  
    &\frac{1}{18K}\sum_{i=1}^{85K} \mu(T^{-i}(B)) \leq \frac{1}{20} + \frac{19}{20}\left(\frac{6}{18}\right)    
\end{align}

For the pattern on the left, we consider indices such that half of the constructed sequence up to that point (considered from $-N-1$) is included in $B$. In $2k$, there are $k$ left out and $k$ included. Similarly, in $6K$, there are $k$ out, $k$ in, $2k$ out, $2k$ in, leaving $3k$ in and $3k$ out. The measures of these are always at least $\frac{19}{40}$ by the same computation as for the cases shown above.The pattern on the right corresponds to adding one more ``block" of left-out nodes, or step in the constructed sequence, to these. The construction of $B$ guarantees that the corresponding values are at most $\frac{1}{20} + \frac{19}{60}$. This constructs two subsequences which must have different limits and contradicts the assumption that the limit converges. 
\end{proof}

\textbf{Remark:} \begin{itemize}
    \item[i)]{ Constructing a measure with either the property in proposition 1 or proposition 2 which is only zero on points known to be in $C \cup D_1$ then this shows that $D_2$ is empty as well.}
    \item[ii)]{This measure has weaker requirements than that posed in \cite{Assanipre}}
    \item[iii)]{This proposition does not require that $T$ be nonsingular, as the following proposition does.}
\end{itemize}

A slight modification of this proposition allows for an argument based on Birkhoff's Ergodic Theorem which generalizes to the class of maps posed in \cite{SyracuseNonsing} as well as to more general measurable maps. The second version also relies on one extra proposition due to Gray and Kieffer \cite{GrayKief}. The proof for Proposition 7 given below may be found in U.Krengel's book \cite{Krengel}. First, recall that for a probability measure space $(\Omega, \mathcal{F}, \nu)$, a measurable map $V: \Omega \to \Omega$ is null-preserving or nonsingular of $\nu(A) = 0$ implies $\nu(T^{-1}(A)) = 0$.

\begin{prop}
    The system $(\Omega, \mathcal{F}, \nu, V)$, for $V$ nonsingular, is asymptotically mean stationary if and only if the averages $\frac{1}{N} \sum_{n=1}^{N} f(V^{n}(x))$ converges $\nu$-almost everywhere for each bounded, measurable function $f$.
\end{prop}
\begin{proof}
    For the reverse direction, note that $\frac{1}{N} \sum_{n=1}^{N} \mathbbm{1}_{A}(T^{n}(x))$ converges for each indicator function $\mathbbm{1}_{A}$ and $A \in \mathcal{F}$. Thus, 
    \begin{equation}
        \int_{\Omega}\frac{1}{N} \sum_{n=1}^{N} \mathbbm{1}_{A}(T^{n}(x)) = \frac{1}{N} \sum_{n=1}^{N} \nu(V^{-n}(A))
    \end{equation}
    converges by the Dominated convergence theorem.

    \medskip

   For the forward direction, assume that $\nu$ is a probability measure. Then, by assumption, $\bar{\nu}(A) = \displaystyle{ \lim_{N \to \infty}} \frac{1}{N} \sum_{n=1}^{N} \nu(V^{-n}(A))$ defines a measure by the Vitali-Hahn-Saks Theorem ($\nu_N (A) = \frac{1}{N} \sum_{n=1}^{N} \nu (V^{-n}(A))$ for all measurable $A$ are absolutely continuous with respect to $\nu$ and converge).

    Consider the set $B_{f}$ of points $\omega$ for which $\frac{1}{N} \sum_{n=1}^{N} f(V^{n}(x))$ converges. The set $B_{f}$ is clearly $V$-invariant and so $\nu(B_f) = \bar{\nu}(B_f)$, where $V$ is measure-preserving with respect to $\bar{\nu}$, so that the Birkhoff-Khinchin theorem gives $\bar{\nu}(B_f) = 1$. 
\end{proof}

\begin{prop}
    Let there exists a finite measure $\mu$ on $\N$ everywhere nonzero such  that $(\N, P(\N), \mu, T)$ is asymptotically mean stationary. Then, $D_2 = \emptyset$. 
\end{prop}
\textbf{Remark:} Note that $T$ is nonsingular with respect to $\mu$ in this case. 
\begin{proof}
The Hopf decomposition in \cite{Assanipre} and \cite{SyracuseNonsing} shows that the set $D_2$ is an at-most countable union of wandering sets, or 
\begin{equation}
    D_2 = \bigcup_{n=0}^{\infty} W_{n}
\end{equation}
where $\mu(T^{i}(W_{n}) \cap T^{j}(W_{n})) = \emptyset$ for $i \neq j$. 

Let $m$ be the restriction to $D_2$ of the limiting measure $\bar{\mu}$ in the proof of the above proposition. Note that $m$ is $T$-invariant. If $\mu(D_2) > 0 $, then $m(D_2) > 0$, and in particular, $m(W_{n}) > 0$ for a wandering set $W_{n}$. Then, $m(\bigcup_{j=0}^{\infty} T^{-j}(W_{n})) \leq m(D_2) < \infty$, but $m(\bigcup_{j=0}^{\infty} T^{-j}(W_{n})) = \sum_{j=1}^{\infty} m(W_{n}) = \infty$, a contradiction.
\end{proof}

This direction of the argument, that the existence of a measure implies something about $D_2$, is the more important direction because it allows for statements on the Collatz map. However, the converse is also true. The following argument immediately generalizes from the Collatz map to any Syracuse-type map such that the preimage of a finite set is finite. 

\begin{lem}
    If $D_2$ is empty, there exists an everywhere-nonzero finite measure $\mu$ asymptotically mean stationary with respect to the Collatz map.
\end{lem}
\begin{proof}

    Recall the decomposition $N = C \cup D_1 \cup D_2$ (where we assume $D_2 = \emptyset$ here). Let $C = \bigcup_{i=1}^{\infty} C_i$ where fore each $i$, $C_i = \{c_1,..,c_{n_i}\}$ is a cycle. 
    
    Then we construct $\mu$ to be a probability measure. Set $\mu(c_1) = ... = \mu(c_{n_i}) = \frac{1}{2^{i+2}n_i}$. Then, for $k \geq 0$, we take the  $ \mu(T^{-k}(T^{-1}(c_j) \backslash C_i)) = \frac{1}{2} \mu(T^{1-k}(T^{-1}(c_j) \backslash C_i))$ where all values in this set have equal measure.  This is to say that we consider the branch of $D_1$ mapping to each $c_j$ in this cycle (without intersecting the cycle elsewhere before mapping to $c_j$) and weight these equally across the $c_j$ values. It is immediate that $\mu(\N) = \sum_{i=1}^{\infty} 2\mu(C_i) = 1$.

    Consider first $A \subset D_1$. Then, $\sum_{j=0}^{\infty} \mu(T^{-j}(A)) < \infty$ implies 
    \begin{equation}
        \displaystyle{\lim_{N \to \infty}} \frac{1}{N} \sum_{j=0}^{N} \mu(T^{-j}(A)) = 0
    \end{equation}
Next, assume that $A$ is a subset of a single cycle $C_i = \{c_1,...,c_{n_i}\}$. We have that
    \begin{align}
    \displaystyle{\lim_{N \to \infty}} \frac{1}{N} \sum_{j=0}^{N} \mu(T^{-j}(A \cap C_i)) = \displaystyle{\lim_{N \to \infty}} \frac{1}{N} \sum_{j=0}^{N} \mu(T^{-j}(\{c_{j_1},...,c_{j_l}\})) \\ 
    = \displaystyle{\lim_{N \to \infty}} \frac{j_l}{N}\left(\sum_{m=0}^{N} \frac{2^{-m}(N-m)}{2^{i+2}n_i} \right) =j_{l}( \sum_{m=0}^{\infty} \frac{2^{-m}}{2^{i+2}n_i}) = \frac{j_{l}}{2^{i+1}n_i} 
\end{align}
The argument immediately generalizes to any subset of $C$ with a bit more algebra in computing the actual limit. Therefore, since any subset of the natural numbers may be decomposed $A = (A \cap C) \cup (A \cap D_1)$, $\mu$ is asymptotically mean stationary with repect to $T$.
\end{proof}

\bigskip

\bibliographystyle{plain}
\bibliography{main}{}

\bigskip
\bigskip
\bigskip

\end{document}